\documentclass[a4paper]{jpconf}

\pdfoutput=1

\usepackage{mathrsfs}

\usepackage{graphicx}
\usepackage{amsthm}
\usepackage{amssymb}
\usepackage{bm}
\usepackage{amstext}
\usepackage{amsmath}
\usepackage{amsfonts}
\usepackage{cite}
\newcommand{\bd}{\begin{definition}}                
\newcommand{\ed}{\end{definition}}                  
\newcommand{\bc}{\begin{corollary}}                 
\newcommand{\ec}{\end{corollary}}                   
\newcommand{\bl}{\begin{lemma}}                     
\newcommand{\el}{\end{lemma}}                       
\newcommand{\bp}{\begin{proposition}}            
\newcommand{\ep}{\end{proposition}}                
\newcommand{\bere}{\begin{remark}}                  
\newcommand{\ere}{\end{remark}}                     

\newcommand{\bt}{\begin{theorem}}
\newcommand{\et}{\end{theorem}}

\newcommand{\be}{\begin{equation}}
\newcommand{\ee}{\end{equation}}

\newcommand{\bit}{\begin{itemize}}
\newcommand{\eit}{\end{itemize}}
\newtheorem{theorem}{Theorem}[section]
\newtheorem{corollary}[theorem]{Corollary}
\newtheorem{lemma}[theorem]{Lemma}
\newtheorem{proposition}[theorem]{Proposition}
\theoremstyle{definition}
\newtheorem{definition}[theorem]{Definition}
\theoremstyle{remark}
\newtheorem{remark}[theorem]{Remark}

\newcommand{\p}{\partial}

\newcommand{\dd}{{\rm d}}
\begin{document}


\title{The representation of spacetime through steep time functions}
\author{E. Minguzzi}

\address{Dipartimento di Matematica e Informatica ``U. Dini'', Universit\`a degli Studi di Firenze,  Via
S. Marta 3,  I-50139 Firenze, Italy}

\ead{ettore.minguzzi@unifi.it}

\begin{abstract}
In a recent work I showed that  the family of smooth steep time functions can be used to recover the order, the topology and the (Lorentz-Finsler) distance of spacetime. In this work I present  the main ideas entering the proof of the (smooth) distance formula, particularly the product trick which converts metric statements into causal ones. The paper ends with a second proof of the  distance formula valid for globally hyperbolic Lorentzian spacetimes.
\end{abstract}

\section{Introduction}
In a recent work \cite{minguzzi17} I  obtained optimal conditions for the existence of steep time functions on spacetime.  The result was used to characterize the Lorentzian submanifolds of Minkowski spacetime and to prove the (smooth Lorentz-Finsler) distance formula. At the meeting  I announced the latter result while placing it into the broad context of functional representation results for topological ordered spaces \cite{nachbin65,minguzzi12d}. In this work  I shall outline the proof strategy instead, by making use of some illustrations, and by leaving the technical details to the original paper.

Unless otherwise stated, we shall work in the general context of closed cone structures and closed Lorentz-Finsler spaces. Let $M$ be a  connected, Hausdorff, second-countable $C^1$ manifold.
A {\em closed cone structure} $(M,C)$, $C\subset TM\backslash 0$, is a closed cone subbundle of the slit tangent bundle such that $C_x:=C\cap T_{x}M$ is a closed (in the topology of $T_xM\backslash 0$), sharp, convex, non-empty cone. The multivalued map $x \mapsto C_x$ turns out to be upper semi-continuous, that is as $y$ approaches $x$, using the identification of tangent spaces provided by any coordinate system, $C_y\backslash C_x\to \emptyset$, cf.\ \cite{aubin84,minguzzi17} for a more precise definition of the last limit. Other useful differentiability conditions on the multivalued map are continuity and local Lipschitzness \cite{fathi12,minguzzi17}. We stress that our cones $C_x$ are not necessarily strictly convex, a feature that will be important, as we shall see.

The closed cone structure conveys the notion of causality. The {\em causal vectors} are the elements of $C$, the {\em timelike vectors} are the elements of $\textrm{Int} C$ while those belonging to $C\backslash \textrm{Int} C$ might be called {\em lightlike vectors}. In the previous expressions $\textrm{Int}$ is the interior for the topology of $TM\backslash 0$. It can be noticed that $ (\textrm{Int} C)_x \subset \textrm{Int} C_x$ and equality holds for $C^0$ cone structures. A continuous causal curve is an absolutely continuous map $x\colon I \to M$ which has causal derivative almost everywhere. The causal relation $J$ consists of all those pairs of events $(p,q)$ such that there is a continuous causal curve connecting $p$ to $q$ or $p=q$.
{\em Closed} cone structures are particularly well behaved since for them the limit curve theorem holds true. Actually, several other non trivial results hold true for closed cone structures, such as the causal ladder of spacetimes. We shall not explore these findings in this work; the reader is referred to \cite{minguzzi17} for more details.

The length of vectors is measured by the fundamental Lorentz-Finsler function $\mathscr{F}\colon C\to [0,+\infty)$ which is positive homogeneous and concave.
 As a consequence, it satisfies the reverse triangle inequality: for every $y_1,y_2\in C$, $\mathscr{F}(y_1+y_2)\ge \mathscr{F}(y_1)+\mathscr{F}(y_2)$.
 In this work a {\em Lorentz-Finsler space} is just a pair $(M,\mathscr{F})$, otherwise called a {\em spacetime}.

 Lorentzian geometry is obtained with the choice
\begin{equation} \label{nia}
\mathscr{F}(y)=\sqrt{-g(y,y)} ,
\end{equation}
where $g$ is the Lorentzian metric. Its cones $C_x$ are said to be {\em round} since the intersection with an affine plane on $T_xM$ is an ellipsoid.

It must be mentioned that in standard Lorentz-Finsler theories the Finsler Lagrangian $\mathscr{L}=-\frac{1}{2}\mathscr{F}^2$ has Lorentzian vertical Hessian $g$ on $C$, $\mathscr{F}(\p C)=0$, and further regularity properties are demanded for $\mathscr{L}$, e.g.\ $g$ is continuous up to $\p C$. The Lorentz-Finsler spaces of this work are much more general, for instance we do not even assume $\mathscr{F}(\p C)=0$. However, we are left with the problem of introducing convenient regularity conditions on $\mathscr{F}$.

One of the key ideas of our work concerns the very definition of {\em closed Lorentz-Finsler space}. This is one instance of application of the {\em product trick}. We observe that $\mathscr{F}$ with its properties allows us to define a sharp convex non-empty cone at every tangent space of $M^\times=M\times \mathbb{R}$, which we give in two versions
\begin{align}
C^\times&=\{(y,z)\colon y\in C\cup \{0\}, \vert z\vert\le \mathscr{F}(y)\} \backslash\{0,0\} \\
C^\downarrow &=\{(y,z)\colon y\in C\cup \{0\}, z\le \mathscr{F}(y)\} \backslash\{0,0\} \label{ana}
\end{align}
The former version would seem the most natural since in Lorentzian geometry it produces a round cone, but it is actually the latter which will prove to be the most useful. Observe that $C^\downarrow$ is not strictly convex.

The main idea is that the cone structure $(M^\times,C^\downarrow)$ encodes all the information on the Lorentz-Finsler space $(M,\mathscr{F})$ so in order to get the best results for $(M,\mathscr{F})$ we have to impose those differentiability conditions on $\mathscr{F}$ which guarantee that $(M^\times,C^\downarrow)$ is a nice type of cone structure. So we impose that $(M^\times,C^\downarrow)$ is a closed cone structure, a condition which is  equivalent to  the upper semi-continuity of both $C$ and $\mathscr{F}$. By definition, with these conditions $(M,\mathscr{F})$ is a {\em closed Lorentz-Finsler space}. At this point, instead of working out new results for Lorentz-Finsler spaces we can just translate the known results for cone structures.
We shall follow this approach in the construction of steep time functions, but in the main paper one can find this idea applied in many directions, from the definition of causal geodesic to the proof of the notable singularity theorems.

A {\em time function} is a continuous function which increases over every causal curve. A {\em temporal function} is a $C^1$ function $t$ such that $\dd t(y)>0$ for every $y\in C$, hence a time function. Let $f\colon C\to [0,+\infty)$ be positive homogeneous. An $f$-steep function is a $C^1$ function $t$ such that
 \begin{equation}
\dd t(y)\ge f(y)
\end{equation}
for every $y\in C$. It is {\em strictly} $f$-steep if the inequality is strict. Since $\mathscr{F}\ge 0$, every strictly $\mathscr{F}$-steep function is temporal. Our main objective is to characterize those closed Lorentz-Finsler spaces which admit a strictly  $\mathscr{F}$-steep function. If $h$ is a Riemannian metric, with some abuse of terminology, we say that a function is $h$-steep if it is $\vert \cdot\vert_h$-steep.

The (Lorentz-Finsler) length of a continuous causal curve $x\colon [0,1]\to M$, $t \mapsto x(t)$, is \[\ell(x)=\int_0^1 \mathscr{F}(\dot x) \dd t\]
(it is independent of the parametrization).   The (Lorentz-Finsler) distance is defined by: for $(p,q) \notin J$,  $d(p,q)=0$, while for $(p,q)\in J$
\begin{equation}
d(p,q)=\textrm{sup}_x \ell(x) ,
\end{equation}
where $x$ runs over the continuous causal curves which connect $p$ to $q$.

It can be observed that if $x\colon [0,1]\to M$ is continuous causal then $x^\times\colon I \to M^\times$ given by
\[
x^\times(t)=(x(t), \int_0^t \mathscr{F}(\dot x(s)) \dd s)
\]
is continuous causal. Let us set  $x(0)=p$, $x(1)=q$, $P=(p,0)$, $Q=(q, \ell(x))$ then $Q$ is the endpoint of the continuous causal curve $x^\times$. Thus $d(p,q)$ is really an upper bound for the fiber coordinate over $(J^\downarrow)^+(P)\cap \pi^{-1}(q)$ where $J^\downarrow$ is the causal relation on  $(M,C^\downarrow)$, and $\pi$ is the projection  $M^\times\to M$. Notice that $d(p,q)$ would be the maximum if $J^\downarrow$ were a closed relation.

We write $C'>C$ if $C'$ is a $C^0$ cone structure such that $ C\subset \textrm{Int} C'$.
We say that $C$ is a {\em causal} closed cone structure if it does not admit closed continuous causal curves.
A {\em stably causal}  closed cone structure is one for which we can find $C'>C$ such that $C'$ is causal. It is by now well established that under stable causality the most useful causal relation is the Seifert relation
\[
J_S=\bigcap_{C'>C} J'
\]
since it is closed and transitive. Under stronger causality conditions, such as causal simplicity or global hyperbolicity $J$ is closed, a fact which implies that $J_S=J$. In fact, under stable causality $J_S$ is really the smallest closed and transitive relation which contains $J$. This result, conjectured by Low in 1996, and implicit in some of Seifert's works in the early seventies was proved by the author in \cite{minguzzi08b} at least for the $C^2$ Lorentzian theory. However, the proof was really  topological so we could generalize it to cone structures  \cite{minguzzi17}.
It must be said that in \cite{minguzzi17} the proof appears after the construction of the steep time functions, thus the order of presentation is reversed with respect to this work. In fact in that work we looked for the most  convenient proofs while here our goal is just that of  showing that our steep time function construction is natural and reasonable given the experience built on Lorentzian geometry.


The cone structure   $C^\downarrow$ is not round, in fact it is not even strictly convex, but it is this cone structure that will prove to be fundamental for our arguments. So it is natural to consider causality theory for non-round cone structures even if one's interest is in Lorentzian geometry.



Among the results which are known to hold in Lorentzian geometry \cite{minguzzi09c} and which survive in the cone structure case \cite{bernard16,minguzzi17} we can find the following: in a stably causal spacetime the Seifert relation can be represented with the set of smooth temporal functions, that is
\begin{equation} \label{rep}
(p,q)\in J_S\Leftrightarrow t(p)\le t(q) \ \textrm{ for every smooth temporal function } \ t.
\end{equation}

Let us return to the product trick. We know that a closed Lorentz-Finsler space is best seen as a cone structure on $M^\times$, so it is natural to ask what happens after  slightly opening the cones on $M^\times$. What is the geometrical meaning of $J^\downarrow_S$, the Seifert relation on $(M^\times, C^\downarrow)$? We have seen that $d(p,q)$ is not the maximum of the fiber coordinate on $(J^\downarrow)^+(P)\cap \pi^{-1}(q)$ just because $J^\downarrow$ is not closed. But $J^\downarrow_S$ is, and moreover it is contained in $J'{}^\downarrow$ for every opening of the cones  $C'{}^\downarrow>C^\downarrow$. Here this opening implies the enlargement of the cone, so $C'>C$, and that of the graphing function $\mathscr{F}$ thus $\mathscr{F}'>\mathscr{F}$. Actually, as a matter of notation, with the latter inequality we shall always include the validity of  the former.

It will therefore not come as a surprise that
\begin{equation} \label{vyv}
J_S^\downarrow=\{((p,r),(p',r'))\colon (p,p')\in J_S \textrm{ and } r'-r \le D(p,p')\}.
\end{equation}
where
the {\em stable distance} $D\colon M\times M\to [0,+\infty]$ is defined as follows. For $p,q\in M$
\begin{equation}
D(p,q)=\mathrm{inf}_{\mathscr{F}'>\mathscr{F}} d'(p,q),
\end{equation}
where $d'$ is the Lorentz-Finsler distance for the  Lorentz-Finsler space $(M,\mathscr{F}')$.

The closure and transitivity of $J_S^\downarrow$ are reflected in two important properties of $D$, namely upper semi-continuity and  reverse triangle inequality: for every $(p,q)\in J_S$ and $(q,r)\in J_S$
\[
D(p,r)\ge D(p,q)+D(q,r).
\]
It can be noticed that $D$ is better behaved than $d$ since the reverse triangle inequality has wider applicability. It can be shown that for globally hyperbolic cone structures $D=d$, but for less demanding causality conditions $D$, rather than $d$, should be regarded as the most convenient Lorentz-Finsler distance.

We say that a closed Lorentz-Finsler space is {\em stable } if it is stable {\em causally} and {\em metrically}, namely if it is stably causal and stably finite. By {\em stably finite} we mean that $d$ remains finite under small perturbations of the Finsler function $\mathscr{F}$, namely  there is $\mathscr{F}'>\mathscr{F}$ such that $d'$ is finite. Clearly, this condition implies that $D<+\infty$, but one of the main result of our work, a bit  lengthy to be discussed here, states that the converse is true: if $D<+\infty$ then $d$ is stably finite \cite{minguzzi17}.

Finally, we mention that globally hyperbolic spacetimes are stable no matter the choice of $\mathscr{F}$ while stably causal spacetimes are conformally stable.

\section{The distance formula}

Our goal is to prove the distance formula. Alain Connes in the early nineties proposed a strategy for the unification of all fundamental physical forces based on non-commutative geometry \cite{connes94}. A key ingredient was the {\em distance formula}, an expression for the Riemannian distance in terms of the 1-Lipschitz functions on spacetime. Unfortunately, the signature of the metric on spacetime is Lorentzian so Connes did not describe correctly the spacetime manifold: there was no causality or Lorentzian distance.

Parfionov and Zapatrin \cite{parfionov00} proposed to consider a more physical Lorentzian version and for that purpose they introduced the notion of steep time function which we already met.
 Let $d$ denote the Lorentzian distance,  and let $\mathscr{S}$ be the family of $\mathscr{F}$-steep time functions. The Lorentzian version of Connes' distance formula would be, for every $p,q\in M$
\begin{equation} \label{dap}
 d(p,q)=\mathrm{inf} \big\{[f(q)-f(p)]^+\colon \ f \in \mathscr{S}\big\}.
\end{equation}
where $c^+=\max\{0, c\}$.

Progress in the proof of the formula was made by Moretti \cite{moretti03} and Franco \cite{franco10} for globally hyperbolic spacetimes. However, in their versions the  functions appearing in (\ref{dap}) were not differentiable everywhere a fact which was a little annoying since in Connes' program the Dirac operator acts on them.
We proved the smooth version for stable spacetimes as part of the next more general theorem on the representability of spacetime by steep time function \cite{minguzzi17}.

\begin{theorem} \label{aas}
Let $(M,\mathscr{F})$ be a closed Lorentz-Finsler space and let $\mathscr{S}$ be the family of smooth strictly  $\mathscr{F}$-steep temporal functions. The Lorentz-Finsler space  $(M,\mathscr{F})$ is stable if and only if $\mathscr{S}$  is non-empty. In this case $\mathscr{S}$   represents
\begin{itemize}
\item[(a)] the order $J_S$, namely $(p, q)\in J_S \Leftrightarrow f(p)\le f(q), \ \forall f \in \mathscr{S}$;
\item[(b)] the manifold topology, namely for every open set $O\ni p$ we can find $f,h \in \mathscr{S}$ in such a way that
$p\in \{q\colon f(q)>0\}\cap \{q\colon h(q)<0\}\subset O$;
\item[(c)] the stable distance, in the sense that  the distance formula holds true: for every $p,q\in M$
\begin{equation}
 D(p,q)=\mathrm{inf} \big\{[f(q)-f(p)]^+\colon \ f \in \mathscr{S}\big\}.
\end{equation}
\end{itemize}
Moreover, {\em strictly} can be dropped.
\end{theorem}
As said above we shall concentrate on the distance formula, i.e.\ point (c). We have also a version for stably causal rather than stable spacetimes. However, in that version the steep functions have to be taken with codomain $[-\infty,+\infty]$ and steep just on the finite set.

Here the main idea is very simple and follows from Eq.\ (\ref{vyv}): the distance formula is a consequence of (\ref{rep}) applied to $(M^\times, C^\downarrow)$, thus as a first step we have to show that this product spacetime is stably causal (we shall return to it later on). At this point the level sets of temporal functions on $(M^\times, C^\downarrow)$ are really local graphs of strictly $\mathscr{F}$-steep functions on $M$. It is at this step that the shape of the cone $C^\downarrow$ becomes essential, since it is thanks to the fact that $C^\downarrow$ contains the fiber direction that the level set is a (univalued) graph, see Figure \ref{mkz}.

 \begin{figure}[ht]
\begin{center}
 \includegraphics[width=14cm]{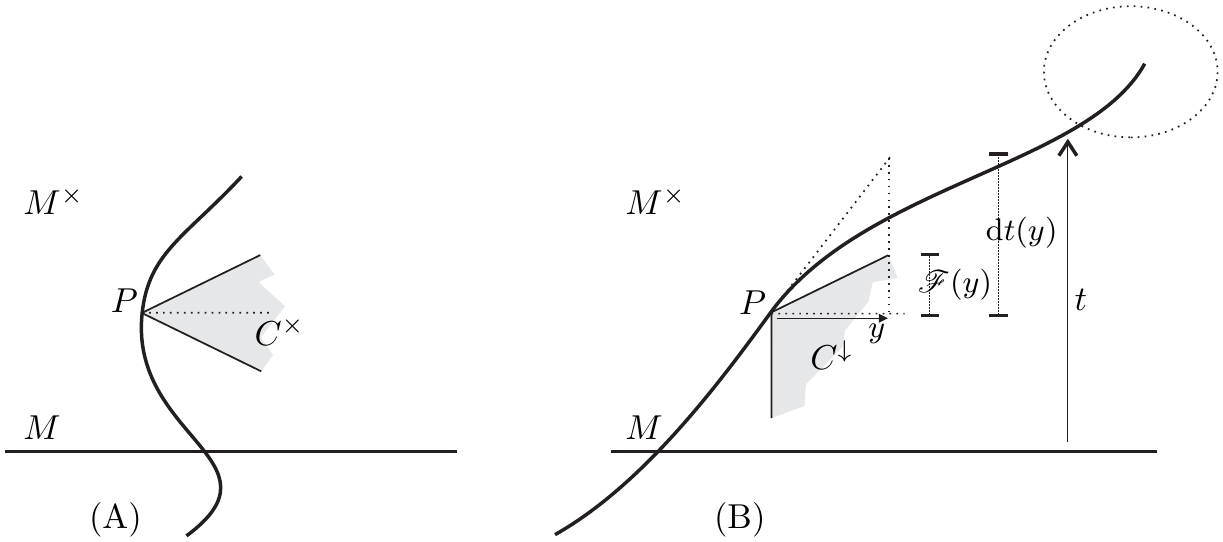}
\end{center}
\caption{The level set of a temporal function $\tau$ on $(M^\times, C^\times)$ and on $(M^\times, C^\downarrow)$. In the latter case, due to the shape of the cone, it is necessarily the local graph of a function $t$ which, furthermore, is strictly $\mathscr{F}$-steep: for every $y\in C$, $\dd t (y)>\mathscr{F}(y)$. The domain of $t$ might not be the whole $M$  since $t$ might blow up. This problem is resolved by constructing $\tau$ directly  through a modification of Hawking's averaging method and by using the condition $D<+\infty$.} \label{mkz}
\end{figure}

Let us see what are the implications of (\ref{rep}) on $(M^\times, C^\downarrow)$.  Let $P=(p,r)$ be a generic point, let $q\in J_S^+(p)$ and let $Q=(q,r')=(q,r+D(p,q))$. By  Eq.\ (\ref{vyv}) $(P,Q)\in J_S^\downarrow$. By (\ref{rep}) if $\tau$ is  temporal on $(M,C^\downarrow)$, $\tau(P)\le \tau(Q)$ which means that $Q$ must stay in the subgraph of the level set of $\tau$ passing through $P$. In other words $t(q)-t(p)\ge r'-r=D(p,q)$, see Fig.\ \ref{sxe}. But for every strictly $\mathscr{F}$-steep function $t\colon M\to \mathbb{R}$ we can find a temporal function $\tau(P)=t(p)-r$ which has the graph of $t$  as level set. Thus the previous argument shows that for every strictly $\mathscr{F}$-steep function $t$, $D(p,q)\le t(q)-t(p)$. We now show that the infimum of the right-hand side is really $D$, thus concluding the proof. In fact (\ref{rep}) states that the temporal functions separate points not belonging to the Seifert relation. Let $Q_\epsilon=(q,r+D(p,q)+\epsilon)$, for $\epsilon>0$, then by  Eq.\ (\ref{vyv}) $(P,Q_\epsilon)\notin J_S^\downarrow$. As a consequence, there is a temporal function $\tau$ on $(M^\times, C^\downarrow)$ which separates them in the sense that $\tau(Q_\epsilon)<\tau(P)$, which implies that the graphing function of the $\tau$-level set passing through $P$ satisfies $t(q)-t(p) \le D(p,q)+\epsilon$, see Fig.\ \ref{sxe}. Due to the arbitrariness of $\epsilon$ we get the distance formula for $(p,q)\in J_S$.

 \begin{figure}[ht]
\begin{center}
 \includegraphics[width=10cm]{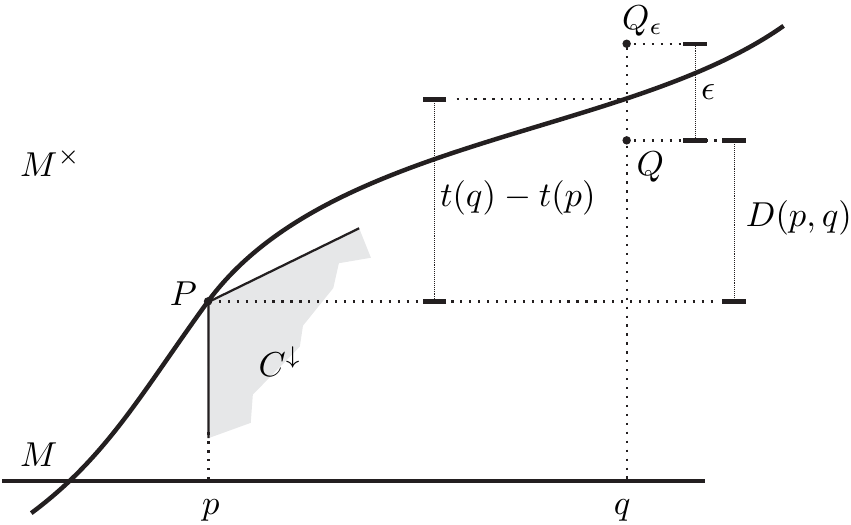}
\end{center}
\caption{Since $(P,Q_\epsilon)\notin J_S^\downarrow$, there is a temporal function $\tau\colon M^\times \to\mathbb{R}$ which separates $P$ and $Q_\epsilon$ in the sense that $\tau(Q_\epsilon)<\tau(P)$. The figure displays the level set of $\tau$ passing though $P$.} \label{sxe}
\end{figure}

Actually, the function $t$ might not be defined everywhere but it turns out that it is sufficient that there exists {\em one} (global) strict $\mathscr{F}$-steep function to prove that such functions separate points not belonging to $J_S^\downarrow$. So we are left with two problems
\begin{itemize}
\item[(i)] We have yet to prove that  $(M,C^\downarrow)$ is stably causal, a fact assumed when we made use of  property (\ref{rep}) on $(M^\times,C^\downarrow)$. This result can be proved constructing directly a time function $\tau$ on $(M,C^\downarrow)$. In fact, we can do this with a sort of Hawking's averaging method.
\item[(ii)] We have  to show that this function $\tau$, suitably smoothed, gives a temporal function on $(M,C^\downarrow)$ whose level sets intersects every $\mathbb{R}$-fiber (so that the graphing function $t$ of a level set  is globally defined).
\end{itemize}
While $(i)$ can be accomplished under stable causality of $(M,C)$, (ii) holds only if $(M,C)$ is stable. In fact, we know that the inequality
$D(p,q)\le t(p)-t(q)$ holds for every strictly $\mathscr{F}$-steep function $t$, thus the existence of just one such function implies that $D$ is finite.

Concerning step $(i)$ we recall that Hawking's averaging method \cite{hawking68,hawking73} applied to $(M^\times,C^\downarrow)$ would consist in the introduction of a positive normalized measure $\mu$ on $M^\times$, absolutely continuous with respect to the Lebesgue measure of any chart, in the introduction of a one-parameter family of cones $C^\downarrow_a{}$, $a\in [0,3]$, $C^\downarrow_a<C^\downarrow_b$, for $a<b$, $C^\downarrow_0=C^\downarrow$, and in the definition of
\[
\tau^\downarrow(P)=-\int_1^2 \mu\big(J_{C^\downarrow_a}^+(P)\big) \dd a .
\]
This construction gives a time function on $(M^\times,C^\downarrow)$ provided this spacetime is stably causal. In fact one would have to choose $C^\downarrow_3$ (stably) causal. Unfortunately, we do not know if $C^\downarrow$ is stably causal. Instead, we construct the function $\tau^\downarrow$ using a different choice for $C^\downarrow_a$ which is causal by construction, namely $C^\downarrow_a$ is built from $(C_a,\mathscr{F}_a)$ according to the analog of Eq.\ (\ref{ana}) where the one-parameter family $(C_a,\mathscr{F}_a)$ is defined for   $a\in [0,3]$, and satisfies: $C_a<C_b$ and $\mathscr{F}_a<\mathscr{F}_b$, for $a<b$, $C_0=C$ and $\mathscr{F}_0=\mathscr{F}$, $C_3$ is stably causal. In other words, the cones $C^\downarrow_a$ do not open in the direction of the $\mathbb{R}$-fiber which for all of them remains lightlike. They are causal because with this definition the projection of a continuous $C_a^\downarrow$-causal curve is a continuous $C_a$-causal curve (unless it coincides with a fiber), and so there cannot be a closed causal curve on $M^\times$ as there would be one on the base, which is impossible as $C_3$ is stably causal.

With this definition $\tau^\downarrow$ is clearly increasing over continuous $C^\downarrow$-causal curves since the  argument of the integral is for every $a$.
What is puzzling is the fact that $\tau^\downarrow$ is continuous despite the fact that the cones are not opened as in Hawking's original prescription. This point cannot be understood intuitively but it is a consequence of two facts: (a) the invariance under the fiber translations of the cone structure $C^\downarrow$, (b) the fact that $C^\downarrow$ projects to a {\em sharp} cone $C$. Notice that the latter property would not hold if $C^\downarrow$ were round (Lorentzian), since the projection would be a half-space.

 Coming to step (ii) the function $\tau$ can be smoothed thanks to a powerful result \cite{minguzzi17} which improves a previous result by Chru\'sciel, Grant and the author \cite[Th.\ 4.8]{chrusciel13}. I recall the theorem though I shall not enter into the details of its use (one can apply it directly to the function $\tau$ on $M^\times$ or to the graphing function $t$ on $M$).
\begin{theorem} \label{moz}
  Let $({ M},C)$ be a closed cone structure and
 let $\tau\colon M\to \mathbb{R}$ be a continuous function. Suppose that there is a $C^0$ proper cone structure $\hat C>C$ and continuous functions  homogeneous of degree one in the fiber $\underline F, \overline F\colon \hat C\to \mathbb{R}$ such that for every  $\hat C$-timelike  curve $x\colon [0,1]\to M$
 \begin{equation} \label{mos}
 \int_x \underline F(\dot x) \dd t\le \tau(x(1))-\tau(x(0))\le \int_x \overline F(\dot x) \dd t.
 \end{equation}
 Let $h$ be an arbitrary Riemannian metric, then for every function $\alpha\colon { M} \to (0,+\infty)$ there exists
a smooth  function $\hat{\tau}$ such that $\vert \hat\tau-\tau\vert <\alpha$ and for every $v\in C$
\begin{equation} \label{kid}
\underline F(v)- \Vert v\Vert_h \le \dd \hat \tau(v) \le \overline F(v)+ \Vert v\Vert_h .
\end{equation}
Similar versions, in which some of the functions $\underline F, \overline F$ do not exist hold true. One has just to drop the corresponding inequalities in (\ref{kid}).
\end{theorem}

More interesting, and peculiar to the proof, is how we solved the problem of constructing $\tau$ in such a way that its level sets intersect every $\mathbb{R}$-fiber. Here the idea comes from Geroch's original construction of the time function on globally hyperbolic spacetime \cite{hawking73}. Basically, we repeat the construction above  with the opposite cones, thus obtaining a time function
\[
\tau^\uparrow(P)=\int_1^2 \mu\big(J_{C^\downarrow_a}^-(P)\big) \dd a .
\]
Then we define $\tau=\log \vert{\tau^\uparrow}/{\tau^\downarrow} \vert$. If we can show that $\tau^\downarrow\to 0$ moving to the future along a fiber (the downward direction in our figures) and that $\tau^\uparrow \to 0$  moving to the past along a fiber (the upward direction in our figures) then $\tau \to +\infty$ in the former limit and $\tau \to -\infty$ in the latter, thus by continuity the level set $\tau=0$ intersects every fiber. The condition $D<+\infty$ guarantees precisely the mentioned limits. Indeed if $K\subset M$ is a compact set then $J_{C^\downarrow_a}^+(P)\cap \pi^{-1}(K)$ will be upper bounded in the extra coordinate, with bound going to $-\infty$ as the fiber coordinate of $P$ goes to $-\infty$. This fact should not totally come as a surprise since as we learned in Fig.\ \ref{sxe}, the function $D$ controls how much we can go up in the extra-coordinate with respect to the starting point.

Our discussion of points (i) and (ii) finishes our exposition of the proof that under the stable condition $D<+\infty$ there are (global) strictly $\mathscr{F}$-steep time functions, which was the missing step in our argument for proving the distance formula.

Incidentally the existence of  $\mathscr{F}$-steep functions is important in the  problem of embedding Lorentzian manifolds into Minkowski spacetime, as shown by M\"uller and S\'anchez \cite{muller11}. As a consequence we also proved a long sought for characterization of the Lorentzian submanifolds of Minkowski spacetime. According to it the Lorentzian submanifolds are precisely the stable spacetimes \cite{minguzzi17}, or more precisely

\begin{theorem}
Let $(M,g)$ be a $n+1$-dimensional Lorentzian spacetime  endowed with a $C^{k}$, $3\le k\le \infty$, metric. $(M,g)$ admits a $C^k$ isometric embedding in Minkowski spacetime $E^{N,1}$, for some $N>0$, if and only if $(M,g)$ is stable.
\end{theorem}

As for the optimal conditions for the validity of the distance formula with $D$ replaced by $d$, i.e.\ that originally suggested by Parfionov and Zapatrin, we found that in Lorentzian manifolds endowed with $C^1$ metrics the formula holds if and only if the spacetime is causally continuous and the Lorentz-Finsler  distance is finite and continuous \cite{minguzzi17}. In particular, globally hyperbolic spacetimes are of this type and in this case the family of functions can be further restricted, so for instance, the functions can be taken Cauchy.



\section{A second proof}

In this section we give a different proof of  the smooth distance formula for Lorentzian globally hyperbolic spacetimes endowed with $C^{2,1}$ metrics. The proof might be adapted to the Lorentz-Finsler case and to  weaker regularity assumptions, however there seems to be no point in pursuing this direction since the strategy of the previous section really allowed us to prove a stronger version  under much weaker conditions.

The approach of this section was really the first one employed by the author to prove the distance formula. Though it has the limitation that it does not give the optimal conditions for the existence of smooth steep time functions,  it  is otherwise perfectly fine if one is interested in globally hyperbolic spacetimes endowed with sufficiently regular metrics.

The proof uses Theorem \ref{moz}, a previous construction of steep functions by the author  \cite{minguzzi16a},  stability results for Cauchy temporal functions  and Lorentzian distance \cite{minguzzi17}, and a previous non-differentiable version of the distance formula proved by Franco \cite{franco10} (thus it is certainly not self contained). For shortness the final step in the proof assumes familiarity with Franco's paper and more generally with
the details of \cite{minguzzi16a,franco10}.  A {\em Cauchy} time function is a time function whose restriction to any inextendible continuous causal curve has image  $\mathbb{R}$. In the next theorems $\mathscr{F}$ is as in Eq.\ (\ref{nia}).

\begin{theorem} \label{spa}
Let $(M,g)$ be a globally hyperbolic spacetime and let $h$ be any Riemannian metric on $M$. There is  a smooth  $h$-steep Cauchy  temporal function $\tau$.
Moreover, there is  a smooth $h$-steep Cauchy  temporal function which is $\mathscr{F}$-steep.
\end{theorem}

This theorem was also obtained by Suhr in a recent work  using different methods \cite[Th.\ 2.3]{suhr15}.

\begin{proof}
 The proof of the existence of a smooth Cauchy $\mathscr{F}$-steep time function in globally hyperbolic spacetimes can be found in \cite{minguzzi16a} (for a previous proof, see \cite{muller11}). One can prove more. Since global hyperbolicity is stable \cite{minguzzi11e,fathi12} one can find a smooth Cauchy steep time function $\tau'$ for a Lorentzian Finsler function $\mathscr{F}'$ (cf.\ (\ref{nia})) with larger cones $C'>C$, such that the indicatrix $\mathscr{F}'{}^{-1}(1)$ of $\mathscr{F}'$ does not intersect that of $\mathscr{F}$. So for every $v \in C_x$,  $\dd \tau'(v) \ge   \mathscr{F}'(v)$ with the bonus that now $\mathscr{F}'$ does not vanish on lightlike vectors, i.e.\ on $\p C_x$.

Let $h$ be any auxiliary Riemannian metric. For every $x\in M$, at $T_xM$ we can shrink through homotheties the indicatrix of $\mathscr{F}'$, so redefining this function but not its cone, in such a way that its indicatrix intersected with $C_x$ is contained in $B^h(x,1)$ (the open unit ball centered at $x$ with respect to the distance induced by $h$). As a consequence, for the Cauchy $\mathscr{F}'$-steep time  function $\tau'$ we have the inclusion $\{v: \dd \tau'(v)=1\}\cap C_x \subset B^h(x,1)$, which implies $\dd \tau'(v) \ge  \mathscr{F}'(v)  \ge \Vert v \Vert_h $ for every $v\in C$. But the fact that the indicatrix of $\mathscr{F}'$ does not intersect that of $\mathscr{F}$ also implies $ \mathscr{F}'(v) \ge   \mathscr{F}(v)$ for every $v\in C$. Thus with the redefinition $\tau' \to \tau$ we get the desired result.
\end{proof}

\begin{theorem} \label{oor}
Let $h$ be an auxiliary Riemannian metric.
In a globally hyperbolic spacetime $(M,g)$ both topology and order can be recovered from the set  $\mathscr{V}$ of smooth Cauchy $h$-steep time  functions. That is:
\begin{itemize}
\item[(a)]
 $(x,y)\in J \Leftrightarrow t(x)\le t(y)$, for every $t\in \mathscr{V}$;
 \item[(b)]  for every open set $O\ni p$ we can find $f,h \in \mathscr{V}$ in such a way that
$p\in \{q\colon f(q)>0\}\cap \{q\colon h(q)<0\}\subset O$.
\end{itemize}
\end{theorem}

\begin{remark} \label{mja}
Thus topology and order can be recovered from the set of smooth  Cauchy $\mathscr{F}$-steep temporal  functions. In fact, if we take $h$ so that its balls do not intersect the indicatrices of $\mathscr{F}$ we have that every smooth Cauchy $h$-steep time  function is a smooth  Cauchy $\mathscr{F}$-steep temporal  function. Notice that $h$ can be chosen complete.
\end{remark}

\begin{proof}
Let $q\notin J^+(p)$ then the spacetime  $N=M\backslash (J^+(p)\cup J^{-}(q))$ is globally hyperbolic. Any Cauchy hypersurface $S$ for $N$ is also  a Cauchy hypersurface for $M$ with $p\in I^+(S)$ and $q\in I^{-}(S)$.
The Geroch topological splitting theorem \cite{hawking73} implies that we can find a Cauchy time function $t$ for $M$ so that $S=t^{-1}(0)$, $t(p)>1$, $t(q)<-1$ (for instance, apply the theorem to $M\backslash J^{-}(S)$ and to $M\backslash J^{+}(S)$  and reparametrize the level sets so obtained into a global Cauchy time function). Let $S_t$ be its constant slices. Inspection of the proof in \cite{minguzzi16a} shows that we can find a smooth Cauchy $\mathscr{F}$-steep temporal function  $t'$ such that  $S'_0:=t'^{-1}(0)\subset t^{-1}([-1,1])$, $t'>t$ for $t>1$, and $t'<t$ for $t<-1$.
An improvement \cite[Th.\ 47]{minguzzi17} over the classical stability result for global hyperbolicity \cite{minguzzi11e,fathi12,samann16,minguzzi17} states that the Cauchy temporal functions are stable so
we can widen the cones while preserving global hyperbolicity and the Cauchy property of $S'_0$. Let $\mathscr{F}'$ be a Lorentzian Finsler function for a wider cone structure, $C'>C$,  chosen so that the indicatrix $\mathscr{F}'{}^{-1}(1)$ intersected with $C$ is contained in the unit ball of $h$, and hence so that any $\mathscr{F}'$-steep
time function $f$ satisfies $\dd f(v)> \Vert v \Vert_h $ for $v\in C$. Repeating the  above argument we find a smooth Cauchy $\mathscr{F}'$-steep temporal function  $t''$ such that  $S''_0\subset {t'}^{-1}([-1,1])$, $t''>t'$ for $t'>1$, and $t''<t'$ for $t'<-1$.  In particular $t''(p)>1>0>-1>t''(q)$. This result proves the representability of the order by smooth $h$-steep Cauchy time functions.


For the topology, let $p\in O$, $O$ open, and let $t$ be a  smooth Cauchy $h$-steep time function such that $t(p)=0$, cf.\ Th.\ \ref{spa}. Let $S_0=t^{-1}(0)$ and let $Q\subset O$ be a compact neighborhood of $p$ such that $Q=J^+(Q)\cap J^{-}(S_0)\cup J^-(Q)\cap J^{+}(S_0)$, and let $\varphi$ be a positive smooth function  supported on $Q$.  Let $\tau^{-}_\varphi$ be the  volume function \cite{chrusciel13}. We know that $\tau^{-}_\varphi$ is $C^1$ with past directed timelike gradient  wherever $E^{-}(q)$ intersect the locus $\varphi>0$ and with a vanishing differential otherwise.
Thus $f:=t+ \tau^{-}_\varphi$ is a $C^1$   Cauchy $h$-steep time function. Notice that the locus $f=0$ coincides with $t=0$ outside $Q$, $f(p)>0$ and $f>0$ only for $t>0$ or inside $Q$. Inverting the time orientation we get a  $C^1$   Cauchy $h$-steep time function $u$ such that $u=0$ coincides with  $t=0$ outside $Q$, $u(p)<0$ and $ u<0$ only for $t<0$ or inside $Q$. Thus $p\in \{q\colon f(q)>0\}\cap \{q\colon u(q)<0\}\subset O$. The smooth version of this inclusion is due to the density of $C^\infty(M,\mathbb{R})$  in
$C^1(M,\mathbb{R})$, see \cite[Th.\ 2.6]{hirsch76}.
\end{proof}

\begin{theorem} \label{nug}
Let $(M,g)$ be a globally hyperbolic spacetime and let $\mathscr{U}$ be the family of smooth, Cauchy, $\mathscr{F}$-steep temporal functions which are $h$-steep for some complete Riemannian metric $h$ (dependent on the function). We have the identity
\begin{equation} \label{dsf}
d(p,q)=\inf\{[f(q)-f(p]^+\colon  f \in \mathscr{U}\}
.\end{equation}
\end{theorem}

\begin{proof}
Suppose first that $(p,q)\notin J$ so that $d(p,q)=0$. We know from Remark \ref{mja} that the functions in $\mathscr{U}$ represent the order, so there is $f\in \mathscr{U}$   such that $f(p)> f(q)$, thus $d(p,q)=0=[f(q)-f(p)]^+$.

Now, suppose that  $(p,q)\in J$.
Let $f$ be smooth and $\mathscr{F}$-steep.
Given a continuous causal curve $x\colon [0,1] \to M$ connecting $p$ to $q$,
\[
f(q)-f(p)=\int_x \dot f  \dd t=\int_x \dd f(\dot x)  \dd t \ge \int_x \mathscr{F}(\dot x)  \dd t= \ell(x).
\]
Thus taking the supremum over the connecting continuous causal curves we get $d(p,q) \le f(q)-f(p)= [f(q)-f(p)]^+$, and taking the infimum over the family of functions we obtain  the inequality $\le$.

For the other inequality we have to show that for every $\epsilon>0$ we can find a smooth Cauchy  $\mathscr{F}$-steep temporal function  $f$ such that $f(q)-f(p) \le d(p,q)+\epsilon$.
We can enlarge the cones while preserving global hyperbolicity. In particular, we can find a Lorentzian $\hat{\mathscr{F}}$ globally hyperbolic, with $\hat C>C$  such that the indicatrices inside the cones do not intersect and $d(p,q) \le \hat{d}(p,q)\le d(p,q)+ \epsilon/4$. The first inequality is clear, while the second inequality follows from \cite[Th.\ 58,59,61]{minguzzi17}.


Let $K$ be a compact neighborhood of $\hat J^+(p)\cap \hat J^{-}(q)$.
Notice that apart from the definition of the cone $\hat C$, so far  $\hat{\mathscr{F}}$ is unconstrained outside $K$. Let $h'$ be a complete Riemannian metric. We choose the indicatrices $\hat{\mathscr{I}}$ of $\hat{\mathscr{F}}$ outside $K$ to be so close to the origin  that their intersection with $C$ is contained in the unit sphere bundle of $h'$, that is $\hat{\mathscr{F}}(v)>\vert v\vert_{h'}$ for every $v\in C$.

Notice that since the indicatrices $\hat{\mathscr{I}}$ and $\mathscr{I}$ do not intersect we can find a Riemannian metric $h$ so small that for every $v\in  C$, $\hat{\mathscr{F}}(v)\ge {\mathscr{F}}(v)+2 \vert v\vert_h$ and outside $K$ the unit balls of $h$ contain those of $h'$ of radius 2, that is $\vert \cdot \vert_{h'}>2\vert \cdot \vert_{h}$ outside $K$.

Let us apply Franco's construction to $(M, \hat{\mathscr{F}})$. Inspection of his proof shows that he constructs a continuous function $\hat f$ such that $\hat f(q)-\hat f(p) <\hat{d}(p,q)+\epsilon/4$ by means of a (locally finite) sum of functions of the form $\hat{d}_r^+(\cdot):=\hat d(r,\cdot)$, $\hat{d}_r^-(\cdot):=-\hat d(\cdot, r)$, where $r$ runs over a countable set of points $\{r_i\}$. Any point of the manifold belongs to the support of one of these functions. As a consequence, for every $(a,b)\in \hat J$, Franco's function satisfies $\hat f(b)-\hat f(a)\ge  \hat\ell(x)$ where $x$ is any $\hat C$-causal curve connecting $a$ and $b$ (just partition the causal curve so that each part belongs to the support of a function $\hat{d}_r^+$ or $\hat{d}_r^-$ and use the reverse triangle inequality for $\hat d$).

We can use Theorem \ref{moz} with $\underline F=\hat{\mathscr{F}}$,  thus we can find $f$ smooth such that $\vert  f-\hat f\vert<\epsilon/4$ and $\dd  f(v) \ge \hat{\mathscr{F}}(v)-\vert v\vert_h\ge \mathscr{F}(v)+\vert v\vert_h$ on every $C$-causal vector $v$. In particular, $f$ is
$h$-steep and
$\mathscr{F}$-steep. But we have also  $\hat{\mathscr{F}}(v)-\vert v\vert_h\ge  \vert v\vert_{h'}-\vert v\vert_h \ge \frac{1}{2} \vert v\vert_{h'}$ outside $K$, thus  $f$  is $h'/4$-steep outside $K$ and  $h$-steep inside $K$, which implies that $f$ is $h''$-steep for some complete Riemannian metric $h''$, and hence Cauchy.
Finally,
\begin{align*}
 f(q)- f(p)&\le  \hat f(q)-\hat f(p)+\vert \hat f(q)-f(q)\vert + \vert \hat f(p)-f(p)\vert \le \hat{d}(p,q)+\tfrac{3}{4}\epsilon\\
&\le  {d}(p,q)+\epsilon.
\end{align*}
\end{proof}

\section*{References}


\providecommand{\newblock}{}

\end{document}